\documentclass[final]{aptpub}
\usepackage{amsmath,amssymb,amstext,color}
\usepackage{graphicx}
\usepackage{personal}
\newcommand{\floor}[1]{\lfloor {#1} \rfloor}

\numberwithin{equation}{section}
\numberwithin{theorem}{section}
\numberwithin{proposition}{section}
\numberwithin{lemma}{section}
\numberwithin{corollary}{section}

\newcommand{\sn}{\sqrt{n}}

\newcommand{\tT}{\tilde T}
\newcommand{\ignore}[1]{}
\newcommand{\To}{\Rightarrow}

\begin{document}
\title{Conditioned Functional Limits and Applications to Queues} 

\authorone[School of Industrial Engineering, Purdue University]{Harsha Honnappa} 
\addressone{315 N. Grant St., West Lafayette IN 47906. honnappa@purdue.edu} 

\authortwo[Department of Electrical Engineering, Computer Science and
Industrial and Systems Engineering, University of Southern
California]{Rahul Jain} 
\addresstwo{3740 McClintock Ave., Los Angeles CA 90089. rahul.jain@usc.edu} 

\authorthree[Marhsall School of Business, University of Southern
California]{Amy R. Ward} 

\addressone{3670 Trousdale Parkway, Los Angeles CA 90089. amyward@marhsall.usc.edu} 

\authornames{Honnappa, Jain, Ward} 
\shorttitle{Conditioned Limits} 

\keywords{renewal process; conditional probability; exchangeable
  increments; functional limits; transitory queues; workload} 
\ams{60K25,60F10}{90B22,68M20}

\begin{abstract}
We consider a renewal process that is conditioned on the number of
events in a fixed time horizon.  We prove that a centered and scaled
version of this process converges to a Brownian bridge, as the number
of events grows large, which relies on first establishing a functional
strong law of large numbers result to determine the centering.  These
results are consistent with the asymptotic behavior of a conditioned
Poisson process.  We prove the limit theorems over triangular arrays
of exchangeable random variables, obtained by conditionning a sequence
of independent and identically distributed renewal processes. We construct martingale difference sequences with respect to these triangular arrays, and use martingale convergence results in our proofs.  To illustrate how these results apply to performance analysis in queueing, we prove that the workload process of a single server queue with conditioned renewal arrival process can be approximated by a reflected diffusion having the sum of a Brownian Bridge and Brownian motion as input to its regulator mapping.
\end{abstract}

\section{Introduction}
The objective of this paper is to prove limit theorems for renewal
processes conditioned on hitting a fixed integer level $n$ in a fixed time horizon; we denote
the corresponding event as $\mathcal A_n$. To
be precise, we establish functional strong law (FSLLN) and functional central
limit theorems (FCLT) as $n$ tends to infinity. A well
known result in stochastic process theory is the ordered
statistics \textsc{(OS)} property of Poisson processes; viz., the
arrival epochs of a homogeneous Poisson process conditioned on
$\sA_n$ are equal in distribution to the ordered statistics of $n$
independent and identically distributed (i.i.d.) uniform random
variables. It follows, for each $n$, that an appropriately defined `conditioned'
Poisson process is equal in distribution to the empirical distribution
process constructed from the ordered statistics. A further corollary to this property is that, appropriately scaled, the conditioned Poisson processes satisfy a
version of Donsker's FCLT for empirical processes so that a sequence of diffusion-scaled conditioned
Poisson processes converges to a Brownian bridge process, as $n\to\infty$. We provide a
proof of this result in Theorem~\ref{thm:conditioned-poisson-fclt}.

As \cite{Fe79,Li1985} show, the only renewal
process that satisfies the~\textsc{OS} property is
Poisson, due to the independent increments property. Nonetheless, a natural conjecture is that a conditioned
renewal process (appropriately scaled) satisfies a counterpart to the Poisson
FCLT in Theorem~\ref{thm:conditioned-poisson-fclt}. The primary
result of this paper, Theorem~\ref{thm:counting},
establishes precisely this result in a triangular array setting. However, the proof is more
subtle and we
construct a sequence of probability sample spaces by conditioning on
the sequence of events
$\{\sA_n~n\geq 1\}$; note that there are no measurability issues owing to the fact
that by definition $\sA_n$ has positive measure. Now, renewal processes display weak dependence, specifically
exchangeability of the inter-arrival times, when conditioned on
the event $\sA_n$. We use the sequence of probability sample spaces to
construct a triangular array of exchangeable random variables
representing the inter-arrival times of a sequence of conditioned
renewal processes.

Consequently, Theorem~\ref{thm:counting} shows that the diffusion-scaled conditioned
renewal process converges weakly to a Brownian bridge process, akin to
Theorem~\ref{thm:conditioned-poisson-fclt}. 
The proof follows from Proposition~\ref{thm:conditioned-fslln}
and Proposition~\ref{thm:conditioned-fclt}, which prove FSLLN and
FCLT's (respectively) for the partial sums constructed from the triangular array. The proof of
Proposition~\ref{thm:conditioned-fslln} follows by the construction of a
triangular array martingale sequence, and then using the
martingale convergence theorem. Similarly, for
Proposition~\ref{thm:conditioned-fclt}, we construct a triangular 
martingale difference array from standardized inter-arrival times, and
show that this martingale difference array satisfies a martingale FCLT.

As an application of the limit results,  in
Section~\ref{sec:workload} we briefly consider the
performance analysis of a
queueing system that sees a fixed finite number of jobs applying for
service over a fixed time horizon. Examples of such systems include clinics, certain call
centers, airline check-in queues, and even certain cloud-based
computing systems where client systems contact a centralized server
for updates. All of these receive a fixed finite number of jobs
over a finite time period. Further,
no jobs are carried over from one `on period' to the next. One approach to modeling
such systems is to use a single server queue with a conditioned
renewal arrival process. This leads to a reflected diffusion
approximation that depends on a Brownian bridge process, which is in
contrast with the conventional heavy-traffic diffusion
approximation that is a reflected Brownian motion. We contrast
these two approximations in Section~\ref{sec:compare}. An important
distinction between the approximations is the fact that we do not
assume an explicit heavy-traffic condition in the conditioned renewal
case.


\paragraph{Existing Literature} \label{sec:lit}
There is a substantial literature related to weak convergence of
conditioned random walks and partial sum processes. In particular, we
note \cite{Li1968,Li1970} where conditioned limit theorems are proved in
some generality for multivariate processes with
i.i.d. increments. \cite{Du1980,Ig1974a,Ig1974b,Ke1992} (and many
others) study the question of the limit behavior of random walks
conditioned to stay positive. There is a much less extensive
literature on conditioned limit theorems for sums of weakly dependent
sequences, which would be relevant to this paper; see \cite{Du1978}
for instance. Conditioned limit theorems have also been used in the
context of performance analysis of queues. For instance \cite{Ho1983}
develops functional limits for the workload process conditioned on the
event that the number of customers in a busy period exceeds or equals
some pre-specified level, as this level tends to infinity. It is shown that the workload process converges to the
Brownian excursion process. The limit results in \cite{As1982} come
closest to the current paper. There, limit theorems for random walks
conditioned on exceeding a certain level in finite time are derived
under the assumption that the random walk has negative drift. It is
shown that the `polygonized' random walk sample path converges to
a Brownian bridge process. This is then used to study the $GI/G/1$
waiting time process in a busy period. 

The diffusion approximation derived for the workload process in
Section~\ref{sec:workload} is similar to that of the $\D_{(i)}/G/1$
queue, derived in \cite{HoJaWa2012}. In the latter model, the arrival
epochs of a large but finite number of arrivals are modeled as
i.i.d. random variables, and the arrival process is defined as the
empirical distribution defined with respect to these random
samples. The workload diffusion approximation is shown to be a
function of a Brownian bridge process, but the resulting limit is more
general than that derived in Section~\ref{sec:workload}. In
particular, it is shown that the reflection is through the so-called
`directional derivative' reflection map \cite[Chapter 9]{Wh01b}. The
primary reason for the difference in the approximations is the fact
that in~\cite{HoJaWa2012} the fluid limit is non-linear (and
time-varying) where as in the current paper, the fluid limit is trivial. More
recently, \cite{BeHoJvL2015} studied the $\D_{(i)}/G/1$ queue under a
heavy-traffic condition on the initial work and exponentially
distributed arrival epochs, and showed that the diffusion
approximation to the queue length is a reflected Brownian motion
process with parabolic drift. The diffusion approximation in
Theorem~\ref{thm:conditioned-poisson-fclt} can also be contrasted with the
diffusion approximation for the $M_t/M/1$ queue implied by the results
in \cite{MaMa95}. 
In particular, a martingale strong approximation argument from \cite[Chapter 7]{EtKu86} 
is used to show that the compensated Poisson process (which is a
martingale) converges to a Brownian motion process. In
 contrast, the conditioned Poisson process is not a compensated
 martingale process, since it is defined with respect to the
 conditioned measure on the set $\mathcal A_n$, and requires a
 different treatment.

\section{Notation}
Let $(\Omega,\sF, \bbP)$ be the sample space with respect to which we
define the random elements of interest. The topology of convergence is
$(\sD,U)$ where $\sD$ is the space of functions that are right
continuous with left limits (\textit{cadlag}) and $U$ is the uniform
metric topology on compact sets of $[0,\infty)$. Weak convergence is represented as
$\Rightarrow$; if necessary, we also note the probability measure as
well: for instance, $\Rightarrow_{\mathbb P}$. Stochastic dominance and equivalence in distribution
are represented as $\leq_{st}$ and $\stackrel{d}{=}$ respectively. 

\section{The Conditioned Poisson Model}~\label{sec:poisson}



\label{sec:poisson}
Let $(P(t), t\geq 0)$ be a unit rate Poisson process defined with respect to $(\Omega, \sF, \bbP)$. Let
$\g : [0,\infty) \to [0,\infty)$ be a non-negative function such that $\G(t) := \int_0^t \g(s) ds$ 
is finite for each $t \in [0,\infty)$. It follows that $P\circ\G$ is a
nonhomogeneous Poisson process with time-varying rate function
$\g$. Fix $ T \in (0,\infty)$, then
\begin{equation} \label{eq:dist-func-from-poisson}
F(t) := \frac{\G(t)}{\G(T)} ~~\forall t \in [0,T],~~
\end{equation}
is a continuous probability distribution function.

\begin{definition}[\textsc{OS} Property] \label{def:OS}
~Let $n \geq 1$. Conditioned on the event $\{P(T) = n\}$, the arrival epochs $(T_1,\ldots, T_n)$, are distributed as the ordered statistics of $n$
independent and identically distributed (i.i.d.) random variables sampled from the
distribution $F(t)$ for $t \in [0,T]$.
\end{definition}

Let  $0 \leq t_1 < \cdots < t_{d} \leq T$ represent an arbitrary
partition of $[0,T]$. Then, the independent increments property of
the Poisson process implies that the following sequence of measures is well-defined:
\begin{align}
  \begin{split}
    \m_n(z_1,z_2,\ldots z_d) := \bbP\bigg(P(t_1) = l_1, ~&P(t_2) -
    P(t_1) = l_2,\\& \ldots, P(t_d)-P(t_{d-1}) = l_{d} \bigg| ~P(T) = n
    \bigg),
  \end{split}
\label{eq:2}  
\end{align}
where $l_i = z_i - z_{i-1}$ (with $z_0 = 0$) and $\sum_{i=1}^d l_i \leq n$. From Kolmogorov's Extension Theorem \cite[Section
A.7]{Du10} it follows that there
exists a stochastic process $\hat P_n$ such that for any $(x_1,\cdots, x_d) \in \bbR^d$,
\begin{align*}
  &\begin{split}
    \bbP \bigg(\hat P_n(t_1) \leq x_1, \ldots, \hat P_n(t_d) \leq x_d
    \bigg) =\bbP \bigg( &\frac{1}{\sn}(P(t_1) - n F(t_1)) \leq
    x_1,\\& \ldots,\frac{1}{\sn}(P(t_d) - n F(t_d) ) \leq x_d \bigg |
    ~P(T) = n\bigg)
      \end{split}\\
    &= \int_{B_n(x_1,\cdots, x_d )} \mu_n \left( dz_1, \ldots, dz_d \right),
\end{align*}
where $B_n(x_1,\cdots, x_d) := \{(z_1,\ldots, z_d) \in \bbZ^d : z_1 \leq n F(t_1) + x_1 \sn, \ldots,
z_d \leq n F(t_d) + x_d \sn\} $. By exploiting the \textsc{OS} property we
can easily obtain a FCLT satisfied by the process $\hat P_n$. Let
$W^0$ be a Brownian bridge process defined with respect to $(\Omega,\sF,\bbP)$.


\begin{theorem} \label{thm:conditioned-poisson-fclt}
~We have,
\(
\hat{P}_n \Rightarrow W^0 \circ F \,\, \text{ in } (\sD, U),
\)
as $n \to \infty$.
\end{theorem}

\begin{proof}
For a fixed $n \geq 1$, $x \in \bbR$ and $t \in [0,T]$ we have
\begin{align*} 
\bbP(\hat P_n (t) \leq x) = \bbP(P(t) \leq x \sn + n F(t) |
P(T) = n).
\end{align*}
Let $\mathbf T := \{T_1, \ldots, T_n\}$ be a collection of i.i.d. random
variables, with distribution function $F$ (defined in
\eqref{eq:dist-func-from-poisson}). Let $A_n(t) := \sum_{i=1}^n
\mathbf{1}_{\{T_i \leq t\}}$ and $\hat A_n(t) := \sn
\left( \frac{A_n(t)}{n} - F(t) \right)$ be the
empirical process associated with $\mathbf T$. The \textsc{OS}
property implies that
\(
\bbP(P(t) = l | P(T) = n) = \bbP(A_n(t) = l),
\)
 using the fact that $\{T_{l} \leq t < T_{l+1}\} = \{P(t)
 = l \}$. Therefore, we have
\begin{eqnarray*}
	\bbP(\hat P_n(t) \leq x) &=& \bbP(A_n(t) \leq x \sn + n
	F(t))\\
	&=& \bbP(\hat A_n(t) \leq x)\\ &\Rightarrow& (W^0 \circ
  F)(t)~\text{as}~n \to \infty
\end{eqnarray*}
the convergence following from Donsker's theorem for empirical
distributions \cite[Chapter 13]{Bi68}, proving the pointwise convergence of the process $\hat P_n$.

Next, consider the partition $0 < t_1 < \cdots < t_d < T$ and observe
that 
\[
\bbP \bigg(\hat P_n(t_1) \leq x_1, \ldots, \hat P_n(t_d) \leq x_d
    \bigg) = \bbP \left( \hat A_n(t_1) \leq x_1,\ldots,\hat A_n(t_d)
      \leq x_d \right).
    \]
From the proof of Donsker's theorem it follows that the increments of the
diffusion-scaled empirical process satisfies $(\hat A_n(t_1), \hat A_n(t_2) -
\hat A_n(t_1)\ldots, \hat A_n(t_d)-\hat A_n(t_{d-1})) \Rightarrow ((W^0
\circ F)(t_1), (W^0 \circ F)(t_2) - (W^0 \circ F)(t_1) \cdots, (W^0
\circ F)(t_d) - (W^0\circ F)(t_{d-1})$, implying
that the finite dimensional distribution of $\hat P_n$ too converges to the
same limit. Next, the tightness of the sequence $\hat A_n$ implies that
$\hat P_n$ is tight. Therefore, by \cite[Theorem 8.1]{Bi68}, $\hat
P_n \Rightarrow W^0 \circ F$ as $n \to \infty$.
$\QED$
\end{proof}



\section{Conditioned Renewal Model}\label{sec:conditioned-renewal}
The proof of Theorem~\ref{thm:conditioned-poisson-fclt} is a
consequence of the \textsc{OS} property. However, \cite[Theorem
1]{Li1985} shows that a renewal process satisfies the \textsc{OS} property if and only if it is Poisson (see
\cite{Fe79} as well). Consequently, it is not \textit{a priori} obvious
that the conditioned renewal process satisfies an analogous result to
Theorem~\ref{thm:conditioned-poisson-fclt}. Furthermore, while we could
argue the existence of the `conditioned' process, $\hat P_n$, by appealing to the
extension theorem and the independent increments property of the
Poisson process, we can no longer do that in the case of a renewal
process. Instead, we prove the conditioned functional limit
theorems in this section by working with the properties of the
inter-arrival times, when conditioned on the event $\sA_n$.

Let $F: [0, \infty) \to [0,1]$ now represent the distribution of a non-negative random variable, with  well
defined density function $f(t) := \frac{d F(t)}{dt}$. Without loss of generality, we also assume that $\int_0^{\infty}
(1-F(t)) dt < \infty$. Recall the
definition of a \textit{finitely exchangeable} sequence \cite[Chapter 9]{Ka2006}.

\begin{definition}[Finitely Exchangeable]\label{def:fe}
~Let $\{X_1, \ldots, X_n\}$ be a collection of random variables defined
with respect to the sample space $(\Omega, \sF, \bbP)$. Then, this
collection is said to be finitely exchangeable if
\(
\{X_1, \cdots, X_n\} \stackrel{d}{=} \{X_{\pi(1)}, \cdots, X_{\pi(n)}\},
\)
where $\pi : \{1, \ldots, n\} \to \{1, \ldots, n\}$ is a permutation
function on the index of the collection.
\end{definition}
Recall that an
\textit{infinitely} exchangeable sequence of random variables
satisfies the permutation condition in Definition~\ref{def:fe} for every
finite subset of random variables in the sequence. Thus, finitely
exchangeable random variables differ from infinitely
exchangeable sequences. Consequently,
important results such as de Finetti's Theorem, which could have been
used to represent the weakly dependent ensemble as a mixture of
independent random variables, are unavailable; see
\cite[Chapter 9]{Ka2006}. In the ensuing discussion, we will refer to
finitely exchangeable collections of random variables as merely
`exchangeable' for brevity. 

Renewal processes satisfy an \emph{exchangeable} (or \textsc{E})
property as summarized in the following lemma.

\begin{lemma}[\textsc{E} Property] \label{prop:exchangeinter}
~Let $\{\xi_{i}, ~i \geq 1\}$ be a sequence of
i.i.d. positive random variables and define $A(t) := \sup \{k > 0 |
\sum_{l=1}^k \xi_{l} \leq t\}, $ for all $t > 0$ to be the associated
renewal counting process. Fix $T \in (0,\infty)$. Then, the collection $\Xi_{n} :=
(\xi_{1}, \ldots, \xi_{n})$ is \textit{finitely exchangeable} when
conditioned by the event $\sA_n = \{A(T) = n\}$.
\end{lemma}

Before proceeding to the proof, note that the event $\mathbb P(\sA_n)
> 0$, under the conditions of the theorem.

\begin{proof}
Let $(x_1, \ldots, x_n) \in [0,\infty)^n$ and consider the measure of
the event $\{\xi_{1} \in dx_1, \cdots, \xi_{n} \in dx_n\} \in \sF$
conditioned on $\sA_n$,
\[
\begin{split}
\mathbb P(\xi_{1} \in dx_1, \cdots, \xi_{n} \in dx_n | \sA_n) = 
\frac{\mathbb P((\xi_{1} \in dx_1, \cdots, \xi_{n} \in dx_n), A(T) =n)}{\mathbb P(\sA_n)}.
\end{split}
\]
Recall that $\{A(T) = n\} = \left\{\sum_{l=1}^n \xi_{l} \leq T < \sum_{l=1}^{n+1} \xi_{l} \right\}$. Now, using the fact that under the measure $\bbP$, $\xi_{i}$ are i.i.d. random variables, it follows that the measure of the joint event is invariant under any permutation of the first $n$ random variables. That is, if $\pi(\cdot)$ is a permutation of $\{1,\ldots,n\}$, then we have
\[
\begin{split}
  \mathbb P \bigg( (\xi_{1} \in dx_1, \cdots, \xi_{n} &\in dx_n), A(T)
  = n \bigg)\\ =~
  &\mathbb P \bigg(\xi_{\pi(1)} \in dx_1, \ldots \xi_{\pi(n)} \in
  dx_n,\sum_{l=1}^n \xi_{\pi(l)} \leq T, \sum_{l=1}^n\xi_{\pi(l)} +
  \xi_{n+1} > T \bigg),
\end{split}
\]
implying that
\[
\mathbb P(\xi_{1} \in dx_1, \cdots, \xi_{n} \in dx_n |\sA_n) = \mathbb P(\xi_{\pi(1)} \in dx_1, \cdots, \xi_{\pi(n)} \in dx_n |\sA_n).
\]
Next, suppose that  $\tilde \pi(\cdot)$ is a permutation of
$\{1,\ldots, n+1\}$. Then, it is possible that $\sum_{i=1}^n \xi_{\tilde\pi(l)} > T$, since $\xi_{n+1} > T - \sum_{l=1}^n \xi_{l} > 0$ conditioned on $\{A(T) = n\}$. Thus, $\Xi_n$ cannot be extended to a larger collection of exchangeable random variables, implying that it is finitely exchangeable.
$\QED$
\end{proof}
Intuitively, the collection $\Xi_n$ is finitely exchangeable owing to the fact that
$\sum_{i=1}^{n} \xi_i \leq T$. This hard bound forces the random
variables to not only take values in a finite interval but to also be
weakly dependent on one another, when conditioned on $\sA_n$.


~Consider $\{(\xi_{n,i},~~i=1,\ldots,n+1),~n \geq 1\}$, a row-wise
independent triangular array of i.i.d. random variables. Define the counting process 
\begin{equation}
  \label{eq:condi-count}
  A_n(t) := \sup \left \{ 0 \leq m \leq n \bigg| S_n\left(\frac{m}{n}
    \right) := \sum_{l=1}^{m} \xi_{n,l} \leq t \right\}.
\end{equation}
By Lemma \ref{prop:exchangeinter}, we know that $\Xi_n := \{\xi_{n,1},\ldots,\xi_{n,n}\}$ is an exchangeable
collection when conditioned on the event $\sA_n := \{\sum_{i=1}^n \xi_{n,i} \leq T <
\sum_{i=1}^{n+1} \xi_{n,i}\} = \{A_n(T) = n\}$. Then.
the collection $\{\Xi_n,~n\geq 1\}$ forms a triangular array of
exchangeable random variables with
independent rows. 

Our analysis will proceed down the triangular array
as $n \to \infty$. Note that the
conditioned probability measure changes for each row of the
array. Classical triangular array results assume that the
array is defined with respect to the \textit{same} probability space. In order
to facilitate the proofs, we first construct a product probability
space that covers the entire array $\{\Xi_n,~n\geq 1\}$.

\subsection{A Product Sample Space}~\label{sec:prod-space}
For a fixed $n \geq 1$ and $T
> 0$, we define the restricted sample space, $(\Omega_n, \sF_n,
\bbP_n)$, where $\Omega_n = \Omega \cap \{A_n(T) = n\}$, $\sF_n := \sigma \{A
\cap \{A_n(T) = n\} : A \in \sF\}$ and $\bbP_n(B) :=
\frac{P(B)}{P(A_n(T)= n)}$ for any $B \in \sF_n$. Clearly $\{\Omega_n,~n\geq 1\}$ forms a partition of $\Omega$. 
Next, we construct a new product space from the restricted sample
spaces $(\Omega_n, \sF_n, \bbP_n)$ as follows. Let, $\bar{\Omega} :=
\Omega_1 \times \Omega_2 \times \cdots$, so that $A \subset
\bar{\Omega} = A_1 \times A_2 \times \cdots$ for sets $A_n \subset
\Omega_n$. The product $\sigma-$algebra, $\bar{\sF} := \sF_1 \otimes
\sF_2 \otimes \cdots$ is the $\sigma-$algebra generated from cylinder sets
of the type $R = \{(\omega_1, \omega_2, \cdots) \in \bar\Omega ~|~
\omega_{i_1} \in A_{i_1}, ~\cdots, ~\omega_{i_k} \in A_{i_k}\}$, where
$(i_1, \ldots, i_l)$ is an arbitrary subset of $\bbN$ of size $k \geq
1$ and $A_{i_n} \in
\sF_n$. The existence of such a product $\sigma-$algebra is well-justified by \cite[Proposition 1.3]{Fo1984}. Finally, we define $\bar{\bbP}(R) = \Pi_{i=1}^k
\bbP_{i_l}(A_{i_l})$, for the cylinder sets. This extends to $\bar{\bbP}
= \bbP_1 \times \bbP_2 \times \cdots$, which is the natural product measure on the
measure space $(\bar{\Omega}, \bar{\sF})$, by standard arguments 
showing that the measure is countably additive on $\bar\sF$. The definition of the
Lebesgue integral on the space $(\bar{\Omega}, \bar{\sF}, \bar\bbP)$ now
follows from standard definitions of integration on product
spaces. However, we introduce some notation to help the following discussion. In
particular, consider a function defined in the following manner: $\bar{X} := X \times \Pi_{l \not = n}
\bbI_{\{\Omega_l\}}$, where $X$ is measurable and
integrable with respect to $(\Omega_n, \sF_n, \bbP_n)$, and
$\bbI_{\{\cdot\}}$ is the indicator function. Then
\[
E_{\bar{\bbP}} [\bar{X}] = \int_{\Omega_n} X d\bbP_n \int_{\Pi_{l \not
  = n} \Omega_l} \bbI_{\{\Omega_l\}} d \bbP_l
\]
is well-defined, and we write this as $E_{\bar\bbP}[X]$, where it is
to be understood that the integration is actually of $\bar X$.

\subsection{Asymptotics of Conditioned Renewal Processes}
Let $\m_n := \bbE[\xi_{n,i} | \sA_n] = \bbE_{\bar
  \bbP}[\xi_{n,i}]$ be the conditional mean of the inter-arrival
times; the exchangeable property implies that these random variables
are identically distributed. Observe that, for a fixed $n$, the
conditioning is with respect to a fixed event $\sA_n$. Therefore,
$\m_n$ is not a random variable.

\begin{lemma} \label{thm:condition-mean-rate}
~The conditional mean $\mu_n$ satisfies

\noindent (i) $\mu_n \to 0$ as $n \to
\infty$,

\noindent (ii) $\text{Var}(\xi_{n,1}) := \bbE[(\xi_{n,1} - \mu_n)^2|A_n(T) =
n] \to 0$ as $n \to \infty$, and

\noindent (iii) $n \mu_n \to T$ as
  $n \to \infty$.
\end{lemma}


Now, consider the
sequence of partial sum processes $\{S_n,~n\geq 1\}$ defined as
$S_n(t) := \sum_{i=1}^{\lfloor nt \rfloor} \xi_{n,i} ~\forall t
\in [0,T]$.  

\begin{proposition}[Partial Sum FSLLN] \label{thm:conditioned-fslln}
~The partial sum sequence satisfies 
\[
~S_n := \frac{S_n}{n} \to \frac{e}{T}~\text{in}~ (\sD,U)~\bar \bbP-\text{a.s.}
\]
as $n \to \infty$.
\end{proposition}

Next, we prove an FCLT for the partial sum sequence $\{S_n,~n\geq
1\}$. Specifically, consider $\{\phi_{n,i}, ~~ l = 1,\ldots,n\}$
defined with respect to $\Xi_n$ as
\[
\phi_{n,i} := \frac{\xi_{n,i} - \mu_n}{\sqrt{n}}.
\]
Following \cite{Mc1974} and \cite[Theorem 24.2]{Bi68}, the following theorem characterizes the sequence $\phi_{n,i}$ and shows
that the partial sums of these random variables converge weakly to a
Brownian bridge process. We assume that the Brownian bridge process
$W^0$ is well-defined with respect to the product sample space $(\bar
\Omega, \bar \sF, \bar \bbP)$.

\begin{proposition} \label{thm:conditioned-fclt} 
~Let $\{(\phi_{n,1},\ldots,\phi_{n,n}),~n \geq 1\}$  be the triangular
  array of random variables defined above and define $\{\hat S_n(t) :=
\sum_{i=1}^{\lfloor n t \rfloor} \phi_{n,i},  t \in [0,T]\}$. Then, \\
\noindent (i) $\sum_{i=1}^n \phi_{n,i} \To_{\bar \bbP} 0$, \\
\noindent (ii) $\max_{1 \leq i \leq n} |\phi_{n,i}| \To_{\bar \bbP} 0$, \\
\noindent (iii) $\sum_{i=1}^n \phi_{n,i}^2 \To_{\bar \bbP} 1$, and\\
\noindent (iv) 
\(
\hat S_n \To_{\bar \bbP} W^0 ~~\text{in}~~ (\sD,U),
 \)
 as $n \to \infty$.
\end{proposition}
The conditions in Proposition \ref{thm:conditioned-fclt} are natural in the context of the
conditioned limit result we seek. Note that the conditioned limit result is akin
to proving a diffusion limit for a tied-down random walk (see
\cite{Li1970,Me1989}). The first condition here enforces a type of
``asymptotic tied down'' property. The second condition is a necessary and
sufficient condition for the limit process to be infinitely divisible
(see \cite{ChTe1958} for more on this). The third condition is
necessary to ensure that the Gaussian limit, when $t = 1$, has
variance $1$. 

Now, define 
and the `inverse' process, corresponding to $S_n$, as
\[
S_n^{-1} (t) := \inf \left\{p \in [0,1] \bigg|  A_n(p) > t \right\},
\]
and the scaled counting process $\bar A_n := n^{-1} A_n$.

\begin{lemma} \label{lem:counting-inverse}
We have,

\noindent (i)
\(
\sup_{0 \leq t \leq T} |\bar A_n(t) - S_n^{-1}(t) | \to 0~\bar \bbP-\text{a.s.}
\)
as $n \to \infty$, and\\
\noindent (ii)
\(
\sn \sup_{0 \leq t \leq T} \left(\bar A_n(t) - S_n^{-1}(t)
\right) \to 0~\bar \bbP-\text{a.s.}
\)
as $n \to \infty$.
\end{lemma}
We can now state and prove the main result of this section, proving
the FSLLN and FCLT for the counting process $A_n$
in~\eqref{eq:condi-count}.

 \begin{theorem} \label{thm:counting}
The counting process $A_n$ satisfies

\noindent (i) $\bar A_n \to \frac{e}{T} ~ \text{in} ~(\sD,U)~\bar \bbP-$a.s. as $n \to \infty$,
where $e : [0,\infty) \to [0,\infty)$ is the identity map, and\\
\noindent (ii) $\sn \left( \bar A_n - \frac{e}{T} \right) \To_{\bar \bbP} -W^0 ~ \text{in}
~(\sD,U)$ as $n \to \infty$, where $W^0$ is the Brownian
bridge limit process observed in Proposition \ref{thm:conditioned-fclt}.
 \end{theorem}
\begin{proof}
Applying \cite[Theorem 7.8.1]{Wh01b}, the FSLLN in Proposition
\ref{thm:conditioned-fslln} implies the convergence of the
corresponding inverse function, $S_n^{-1}$ to $e$. Part
(i) of Lemma \ref{lem:counting-inverse}, in turn, implies the
convergence of the counting process $\bar A_n$. This proves part (i).

Next, 
\cite[Theorem 7.8.2]{Wh01b} and the FCLT in Proposition~\ref{thm:conditioned-fclt}~(iv)
implies that $\sn (S_n^{-1} - e T^{-1}) \Rightarrow -
W^0$. Part~(ii) of the theorem follows from Lemma \ref{lem:counting-inverse}~(ii).
$\QED$ 
\end{proof}


We now proceed to the proofs of the lemmas and propositions.
\begin{proof}[Proof of Lemma~\ref{thm:condition-mean-rate}]
\noindent~$(i)$ The conditional intensity function (CIF) of $A_n(t)$ is $\l^*(t)dt
  := \bbE[A_n(dt)| \mathcal{H}_t] = \frac{f(t) dt}{1 - F(t)} \geq 0$,
  where $\mathcal{H}_t$ is the filtration generated by $A_n$. Let
  $\L^*(t) = \int_0^t \l^*(s) ds$ be the integrated CIF, so that
  $A_n(t) - \L^*(t)$ is a compensated Martingale
  process. \cite[Theorem 7.4.1]{DaVe2003a} shows that $\tilde A_n(t) =
 A_n(\L^{*^{-1}}(t))$ is a unit rate Poisson process. That is, if
  $\{T_1, T_2, \ldots\}$ is a realization of the event epochs of process $A_n$,
  then $\{\tT_i = \L^*(T_i)\}$ is equal (in distribution) to a
  realization from a unit rate Poisson process. As $\L^*$ is
  non-decreasing, it follows that $\{\tT_{n+1} > \L^*(T) \geq \tT_n\}$
  if and only if $\{T_{n+1} > T \geq T_{n}\}$. \cite[Theorem 7.4.1]{DaVe2003a} implies that $\{
  A_n(T) = n\} = \{\tilde A_n(\L^*(T)) = n\}$.

Now, we have
\[
\bbP(\xi_{n,1} > u |  A_n(T) = n) = \bbP(\phi_1 \geq \L^*(u) | \tilde A_n(\L^*(T)) = n),
\]
where $\phi_1 := \L^*(\xi_{n,1})$. Recall that a Poisson process
satisfies the \textsc{OS} property in Definition~\ref{def:OS}. It follows that
\[
\bbP(\phi_1 \geq \L^*(u) | \tilde A_n(\L^*(T)) = n) = \left( 1 - \frac{\L^*(u)}{\L^*(T)} \right)^n.
\]

Now, by definition
\begin{eqnarray}
  \nonumber
  \mu_n &=& \int_0^T \mathbb P(\xi_1 > u | A_n(T) = n) du\\
  \label{eq:mu-integral}
  &=& \int_0^T \left(1 - \frac{\Lambda^*(u)}{\Lambda^*(T)} \right)^n du.
\end{eqnarray}
Since $\Lambda^*(t)$ is a non-decreasing function of $t$, it follows
that the integrand in \eqref{eq:mu-integral} is bounded above by
$1$ for all $n \geq 1$. Furthermore, for every $t \in (0,T]$, $\lim_{n \to \infty}
\left( 1 - \frac{\Lambda^*(t)}{\Lambda^*(T)} \right)^n =
0$. Therefore, by the bounded convergence theorem it follows that
$\lim_{n \to \infty} \mu_n = 0$.

\noindent~$(ii)$ Observe that, by definition,
\(
\int \xi_{n,1}^2 d \bbP_n \leq T^2 ~\forall~n \geq 1,
\)
implying $\Xi_n$ is a uniformly integrable (U.I.) family of random
variables. Then, fixing $\epsilon > 0$, part $(i)$ of the lemma implies that $\bbP_n(\xi_{n,1} >
\e) \to 0$ as $n \to \infty$. We now have
\begin{align*}
  \int \xi_{n,1}^2 d \bbP_n &= \int_{\{\xi_{n,1}> \epsilon\}} \xi_{n,1}^2
                              d \bbP_n + \int_{\{\xi_{n,1} \leq \epsilon\}} \xi_{n,1}^2
                              d \bbP_n\\
                            &\leq T^2 \bbP_n\left( \xi_{n,1} >
                              \epsilon\right) + \epsilon \int
                              \xi_{n,1} d\bbP_n,
\end{align*}
implying that
\(
\int \xi_{n,1}^2 d \bbP_n \to 0
\)
as $n \to \infty$. Thus, $\text{Var}(\xi_{n,1}) \to 0$ as $n \to \infty$.

\noindent~$(iii)$ Let $S_n := \sum_{i=1}^n \xi_{n,i}$ and
$S_{n+1} = S_n + \xi_{n,n+1}$. Since the random variables $\xi_{n,i}$ are identical in distribution,
\[
\m_n = \bbE[\xi_{n,1} | A_n(T) =n ] = \frac{1}{n} \bbE[S_n | A_n(T) = n].
\]
By definition $\{A_n(T) = n\} =
\{S_n \leq T < S_{n+1}\}$, implying that
\(
  \bbE[S_n | A_n(T) = n] \leq T
\)
for all $n \geq 1$. Next, fix $\epsilon > 0$, and note that
\begin{eqnarray*}
\bbE [S_n | A_n(T) = n] &=& \frac{\bbE[S_n \mathbf 1_{\{S_n \leq T <
    S_{n+1}\}}]}{\mathbb P(A_n(T) = n)}\\
  & \geq & (T - \epsilon) \frac{\mathbb P(T-\epsilon < S_n \leq T <
           S_{n+1})}{\mathbb P(S_n \leq T < S_{n+1})}.
\end{eqnarray*}
Now, consider the partition of $\{S_n \leq T < S_{n+1}\} =
\{T-\epsilon < S_n \leq T < S_{n+1}\}
\cup \{S_n < T-\epsilon, S_{n+1} > T\}$. We have that, under $\mathbb P$,
\[
\mathbb P(S_n < T-\epsilon, S_{n+1} > T) \leq \mathbb P(S_n < T).
\]
By the strong law of large numbers, it follows that $S_n \to
\infty$ $\bbP$-a.s. as $n \to \infty$. Therefore, $\mathbb P(S_n < T) \to 0$
as $n \to \infty$. Thus, it follows that
\[
\lim_{n \to \infty} \frac{\mathbb P(T-\epsilon < S_n \leq T <
           S_{n+1})}{\mathbb P(S_n \leq T < S_{n+1})} = 1.
\]
Therefore, 
\begin{equation} \label{eq:2}
\liminf_{n \to \infty} \bbE [S_n | A_n(T) = n] \geq (T-\epsilon).
\end{equation}
Since $\epsilon > 0$ is arbitrary, it follows that $\lim_{n \to
  \infty} n \mu_n = \lim_{n \to \infty} \bbE[S_n|A_n(T) = n] = T$.

$\QED$
\end{proof}

  

\begin{proof}[Proof of Proposition~\ref{thm:conditioned-fslln}]
~Without loss of generality let $T = 1$. Consider, for $t\in[0,1]$,
\[
\left | \sum_{l=1}^{\lfloor nt \rfloor} \xi_{n,l} - t \right | \leq
\left |\sum_{l=1}^{\lfloor nt \rfloor} (\xi_{n,l} - \m_n) \right | + \left |\lfloor nt \rfloor \mu_n - t \right|.
\]
The second term on the RHS tends to $0$, as a consequence of Lemma
\ref{thm:condition-mean-rate}~(iii). 
Define the martingale sequence,
\(
z_{n,l} := (\xi_{n,l} - \m_n) - \bbE_{\bar \bbP}[(\xi_{n,l} - \m_n) | \sF_{n,l-1}],
\)
 where $\sF_{n,l} := \s \{(\xi_{n,1} - \m_n), \ldots, (\xi_{n,l-1} -
 \m_n), \sum_{i=l}^n (\xi_{n,i} - \m_n)\}$. Note that expectation is
 taken with respect to the measure $\bar \bbP$, implying that there is (implicitly)
 a conditioning with respect to the event $\sA_n$ as well.

 It follows that
\begin{align*}
\sum_{i=j}^n (\xi_{n,i} - \m_n) &= \sum_{i=j}^n \bbE_{\bar \bbP}[(\xi_{n,i} - \m_n) |
\sF_{n,j-1}]\\ &= (n-j+1) \bbE_{\bar \bbP}[\xi_{n,j} - \m_n | \sF_{n,j-1}],
\end{align*}
where the last equality follows from the fact that the random
variables are exchangeable (and hence identically distributed) under
the measure $\bar \bbP$. This implies that
\begin{align*}
z_{n,l} &= (\xi_{n,l} - \m_n) - \frac{1}{n-l+1} \sum_{i=l}^n (\xi_{n,i}
- \m_n)
\end{align*}
 Using the fact that $\xi_{n,l} \in [0,1]$, under the measure $\bar
 \bbP$, it follows that
\begin{eqnarray}
\nonumber
\sum_{l=1}^n \frac{1}{n-l+1} \sum_{i=l}^n \xi_{n,i} &=&
\sum_{l=0}^{n-1} \frac{1}{n-l} \sum_{j=l+1}^n \xi_{n,j}\\
\nonumber
&\leq& \left( \sum_{l=0}^{n-1}\frac{1}{n-l}
       \left(n-l-1\right)\right)\\
\nonumber
&\leq& n,
\end{eqnarray}
and consequently,
\[
\sum_{l=1}^n (\xi_{n,l} - \m_n) \leq \sum_{l=1}^n z_{n,l} + n.
\]
On the other hand, observe that
\begin{align*}
\sum_{l=1}^n (\xi_{n,l} - \m_n) &= \sum_{l=1}^n z_{n,l} +
  \sum_{l=1}^n \frac{1}{n-l+1} \sum_{i=l}^n \xi_{n,i} - n \m_n\\
  &\geq \sum_{l=1}^n z_{n,l} - n,
\end{align*}
where we have used the fact that $\m_n \leq 1$ in the final
inequality. Now, fix $\e > 0$ and use the inequalities above to obtain 
\(
\left \{\omega \in \bar \Omega : \left|\sum_{l=1}^n (\xi_{n,l} - \m_n
  )\right| > \epsilon \right\}
\)
\begin{align*}
&= \left\{ \omega \in \bar \Omega :
                                 \sum_{l=1}^n (\xi_{n,l} - \m_n) > \e
                                 \right\} \cup \left\{ \omega \in \bar \Omega :
                                 \sum_{l=1}^n (\xi_{n,l} - \m_n) < -\e
                                 \right\} \\
 &\subseteq \left \{\omega \in \bar \Omega : \sum_{l=1}^n z_{n,l} +n >
   \e \right\} \cup \left\{\omega \in \bar \Omega : \sum_{l=1}^n z_{n,l} <
   n -\e \right\}.
\end{align*}
Since $|z_{n,l}| \leq 2$, the Azuma-Hoeffding inequality implies that
\[
\bar \bbP \left(\sum_{l=1}^n z_{n,l} > \e - n \right) \leq \exp \left(- \frac{\left(\e - n \right)^2}{8n} \right),
\]
and
\[
\bar \bbP \left(\sum_{l=1}^n z_{n,l} < -\e \right) \leq \exp \left(- \frac{(n-\e)^2}{8n} \right).
\]
Therefore, using the union bound
\[
\bar \bbP \left( \left|\sum_{l=1}^n \xi_{n,l} - \m_n \right| >
  \e\right) \leq 2 e^{-(n-\e)^2/8n}.
\]


Now, by Cauchy's ratio test, it can
be readily verified that for any $\e > 0$
\[
\sum_{n=1}^{\infty} \bar \bbP \left( \left|\sum_{l=1}^n(\xi_{n,l} -
    \m_n) \right| > \e \right) < \infty.
\]
Thus, by the First Borel-Cantelli Lemma, $\bar
\bbP(|\sum_{l=1}^n(\xi_{n,l} - \m_n)| > \e ~ \text{i.o.} ) = 0$. Therefore,
$\sum_{l=1}^n(\xi_{n,l} - \m_n) \to 0 ~\bar \bbP-$a.s. as $n \to
\infty$. Clearly, this holds for any $t \in [0,1]$, so that
\(
\sum_{l=1}^{\lfloor nt \rfloor}(\xi_{n,l} - \m_n) \to 0~\bar\bbP-
\)
a.s. as $n \to \infty$. The proof of uniform convergence on $[0,1]$
follows from standard arguments (see \cite[Chapter 5]{ChYa01} for instance).
$\QED$ 
\end{proof}

\begin{proof}[Proof of Proposition~\ref{thm:conditioned-fclt}]
First, note that the exchangeability of $\{\phi_{n,i}\}$ follows directly from that of
$\{\xi_{n,i}\}$. 

\noindent $(i)$ The proof follows by using the definition of $\bar \bbP$. Fix $\e > 0$, and consider 
\begin{align*}
\bar \bbP\left( \left|\sum_{l=1}^n \phi_{n,l} \right| > \e \right) =&~
                                                                      \bar
                                                                     \bbP
                                                                      \left(
                                                                      \left|\sum_{l=1}^n
                                                                      \phi_{n,l}
                                                                      \right|
                                                                      >
                                                                      \e,
                                                                      A_n(T)
                                                                      =
                                                                      n
                                                                      \right)\\
=&~\bar\bbP\left( \left|\sum_{l=1}^n \xi_{n,l} - n \mu_n\right| > \e
  \sqrt{n}, A_n(T) = n \right)\\
=&~\bar\bbP\bigg(\sum_{l=1}^n \xi_{n,l} > \e \sqrt{n} + n \mu_n ,
      A_n(T) = n \bigg) \\ &\qquad+ \bar\bbP\left(\sum_{l=1}^n \xi_{n,l} < -\e
      \sqrt{n} + n \mu_n , A_n(T) = n\right).
\end{align*}

Recall that $\left\{A_n(T) = n \right\} = \left\{\sum_{l=1}^n
  \xi_{n,l} \leq T < \sum_{l=1}^n \xi_{n,l} + \xi_{n,n+1} \right\}$.
It follows that for any $\omega \in \sB_n := \left\{\sum_{l=1}^n
  \xi_{n,l} > \e \sqrt{n} + n \mu_n , A_n(T) = n) \right\}$ we have
\(
T \geq \sum_{l=1}^n \xi_{n,l} > \e \sqrt{n} + n \mu_n
\)
and that $n \mu_n = E_{\bar \bbP}[\sum_{l=1}^n
\xi_{n,l}] \leq T$. Therefore, $n
\mu_n$ is uniformly bounded (for every $n \geq 1$). Then,
for a given $T$, there exists a $n_T$ such that for every $n > n_T$,
$\sqrt{n} \e + n \mu \geq T$. As $\e > 0$ is arbitrary,
asymptotically, $\sB_n$ is an impossible event. 

Next, consider the event $\sC_n :=
\{\sum_{l=1}^n \xi_{n,l} < -\e \sqrt{n} + n \mu_n , A_n(T) =
n\}$. Using the facts that $\xi_{n,l} \geq 0$ and $n
\mu_n \leq T$ for all $n$, we have
\(
- \e \sqrt{n} + T \geq -\e \sqrt{n} + n \m_n > \sum_{l=1}^n \xi_{n,l} \geq 0.
\)
Clearly, as $n \to \infty$, $-\e \sqrt{n} + T \to -\infty$ implying
that $-\e \sqrt{n} + n \m_n \to -\infty$. Since $\e > 0$ is arbitrary,
for large enough $n$ $\sC_n$ too is an impossible event. It follows that
$\phi_{n,l} \To_{\bar \bbP} 0$ as $n \to \infty$.

\noindent $(ii)$ First, for a fixed $\e > 0$ the union bound implies that
\begin{align*}
\bar \bbP \left(\max_{1 \leq l \leq n} \left|\phi_{n,l} \right| > \e
\right) &\leq \sum_{l=1}^n \bar \bbP\left( \left|\phi_{n,l} \right| >
  \e \right)\\
&\leq n \bar \bbP \left( \left|\phi_{n,1} \right| > \e \right)\\
&\leq n \frac{E_{\bar \bbP} \left|\xi_{n,l} - \m_n \right|^2}{n\e^2 } = \frac{\text{Var}(\xi_{n,1})}{\e^2},
\end{align*}
where the latter expression follows by an application of Chebyshev's
inequality under the $\bar \bbP$ measure.~
Lemma~\ref{thm:condition-mean-rate} implies that $\text{Var}(\xi_{n,1}) \to 0$ as $n \to \infty$. As $\e > 0$ is arbitrary, (ii) is proved.

\noindent $(iii)$ Define the martingale difference sequence $Z_{n,l} := \phi_{n,l}^2 - E_{\bar \bbP}[\phi_{n,l}^2 | \sF_{n,l-1}]$, where $\{\sF_{n,l}\}$ is a filtration defined with respect to $\phi_{n,l}^2$ as 
\(
\sF_{n,l} = \sigma \left( \phi_{n,1}^2, \ldots, \phi_{n,l-1}^2,
  \sum_{i=l}^n \phi_{n,i}^2 \right). 
\)

Now, consider the conditional expectation in the definition of $Z_{n,l}$. Notice that we have,
\begin{eqnarray*}
\sum_{i=j}^n \phi_{n,i}^2 &=& E_{\bar \bbP} \left[\sum_{i=j}^n
                              \phi_{n,i}^2 | \sF_{n,j-1} \right]\\
&=& E_{\bar \bbP} \left[\sum_{i=j}^n \phi_{n,j}^2 | \sF_{n,j-1} \right]\\
&=& (n-j+1) E \left[\phi^2_{n,j} | \sF_{n,j-1} \right].
\end{eqnarray*}
The penultimate equation follows from the fact that  $\phi_{n,l}^2$
are exchangeable, and the last by the fact that they are also
identically distributed. It follows that
\[
E_{\bar \bbP}\left[\phi^2_{n,j} | \sF_{n,j-1} \right] = \frac{1}{n-j+1} \sum_{i=j}^n \phi_{n,i}^2,
\]
and
\[
Z_{n,l} = \phi_{n,l}^2 - \frac{1}{n-l+1} \sum_{i=l}^n \phi_{n,i}^2.
 \]
Now, by definition $\phi_{n,l} \leq 2T/\sqrt{n}$ under the measure
$\bar \bbP$, so that

\begin{eqnarray*}
\sum_{l=1}^n \frac{1}{n-l+1} \sum_{i=l}^n \phi_{n,i}^2  &=& \sum_{l=0}^{n-1} \frac{1}{n-l} \sum_{j=l+1}^n \phi^2_{n,j}\\
&\leq& \frac{4T^2}{n} \sum_{l=0}^{n-1} \left( 1 - \frac{1}{n-l}
\right) \\
&\leq& 4T^2.
\end{eqnarray*}
Thus, we have
\(
\sum_{l=1}^n Z_{n,l} \geq \sum_{l=1}^n \phi_{n,l}^2 - 4T^2.
\)
Fix $\e > 0$, and use the Azuma-Hoeffding inequality to obtain
\begin{align*}
\bar \bbP\left( \sum_{l=1}^n \phi_{n,l}^2 - 1 \geq \e \right) &\leq
  \bar \bbP \left(\sum_{l=1}^n Z_{n,l} \geq \e +1 - 4T^2 \right)\\
& \leq \exp \left( - \frac{(\e +1 - 4T^2)^2}{n \times \frac{64 T^4}{n^2}} \right),
\end{align*}
where the bound in the numerator on the R.H.S. follows by the facts that
\[
|Z_{n,l}| \leq |\phi_{n,l}^2| + \frac{1}{n-l+1} \sum_{j=l}^n |\phi_{n,j}^2| \leq 2 |\phi_{n,l}^2|,
\]
and $\phi_{n_l}^2 \leq 2 T^2/n$. It follows that
\[
\bar \bbP\left(\sum_{l=1}^n \phi_{n,l}^2 - 1 \geq \e \right) \to 0 ~\text{as}~n\to\infty.
\]

Next, since $\phi_{n,l}^2 \geq 0$ for all $l \leq n$, it follows that $\sum_{l=1}^n Z_{n,l} \leq \sum_{l=1}^n \phi_{n,l}^2$. Clearly,
\[
\bar \bbP \left(\sum_{l=1}^n \phi_{n,l}^2 < 1 - \e \right) \leq \bar
\bbP \left(\sum_{l=1}^n Z_{n,l} < 1 - \e\right). 
\] 
Using the Azuma-Hoeffding inequality again, we have
\[
\bar \bbP \left(\sum_{l=1}^n Z_{n,l} < 1 - \e \right) \leq \exp\left(-\frac{(1-\e)^2}{n \times \frac{64 T^4}{n^2}} \right),
\]
implying that
\[
\bar \bbP \left(\sum_{l=1}^n \phi_{n,l}^2 < 1 - \e \right) \to 0 ~\text{as}~n\to\infty.
\]
Finally, it
follows that $\sum_{l=1}^n \phi_{n,l}^2 \To_{\bar\bbP} 1$ as
$n \to \infty$.

\noindent $(iv)$ Parts $(i),~(ii),~(iii)$ verify \cite[Theorem
24.2]{Bi68}, implying that $\hat A_n
\Rightarrow_{\bar \bbP} W^0$ in $(\sD,U)$ as $n \to infty$.
$\QED$
\end{proof}

\begin{proof}[Proof of Lemma~\ref{lem:counting-inverse}]
  \noindent~$(i)$ Fix $t \in [0,1]$. By definition it follows that
  $S_n (\bar A_n(t)) \leq t$ and
  $S_n (S_n^{-1}(t)) > t$ (and
  $S_n (S_n^{-1}(t)-) \leq t$). Thus, for any $\e > 0$,
  $S_n (\bar A_n(t) + \e) > t$. In particular,
  $\bar A_n(t) + \frac{1}{n} \geq S_n^{-1}(t)$. Since
  $S_n$ is non-decreasing (since the increments
  $\xi_{n,l} \geq 0$), it follows that
  \[
  \frac{1}{n} \geq S_n^{-1}(t) - \bar A_n(t) \geq 0,
  \]
  where the last inequality follows by definition. 

  \noindent~$(ii)$ The result is an obvious corollary of the argument for
  part $(i)$.
$\QED$
\end{proof}

\section{An Application to Transient Workload Analysis}~\label{sec:workload}
We now demonstrate how the conditioned limit theorems developed in the
previous section can be used to conduct a transient performance
analysis of a single server queue. We will focus on the workload
process, though the analysis can be extended to other performance
metrics as well. 

The `data' of the queueing model are as follows: Let $T = 1$ (with out
loss of generality) and consider a triangular
array of tuples, $\left \{\left((\xi_{n,1}, \nu_{n,1}),\ldots,
  (\xi_{n,n+1}, \nu_{n,n+1}) \right),~n\geq
1 \right\}$. For simplicity we will assume that $\xi_{n,i}$ and $\nu_{n,i}$
are independent for all $i = 1,\ldots,n$. We will also assume that
$\xi_{n,i}$ are identically distributed, and that
the unconditional mean satisfies $\bbE[\xi_{n,i}] < \infty $, for all $n
\geq 1$ and $i \leq n$. Similarly, for $\nu_{n,i}$ we assume
i.i.d. random variables with $\bbE[\nu_{n,i}] = 1$ and variance
$\sigma^2$ for all $n \geq 1$.

As in Section~\ref{sec:conditioned-renewal}, we will focus on the
sub-array, \[\left\{\Xi_n,~n\geq 1\} =\{\left((\xi_{n,1},\nu_{n,1}),\ldots,(\xi_{n,n},\nu_{n,n})\right),~n\geq
1 \right\},\] and define the corresponding conditioned measures
$\{\bbP_n,~n\geq 1\}$ and the joint distribution $\bar \bbP$ as in Section~\ref{sec:prod-space}. Observe that the independence of $\{\xi_{n,l},~l\leq n\}$ and
  $\{\nu_{n,l},~l\leq n\}$ implies that $\bbP_n \left( (\xi_{n,1}, \ldots, \xi_{n,n}) \in d \mathbf x,
    (\nu_{n,1},\ldots,\nu_{n,n}) \in d \mathbf z) \right)$
  \begin{align*}
  &= \bbP (\left( (\xi_{n,1}, \ldots, \xi_{n,n}) \in d \mathbf x,
    (\nu_{n,1},\ldots,\nu_{n,n}) \in d \mathbf z) \right) | \sA_n)\\
    &=\bbP_n \left( (\xi_{n,1}, \ldots, \xi_{n,n}) \in d \mathbf x
      \right) \bbP \left((\nu_{n,1},\ldots,\nu_{n,n}) \in d
      \mathbf z) \right)
  \end{align*}
where $\mathbf x,~\mathbf z \in \bbR^n$. Recalling the definition of
$\bar \bbP$, we note that 
\(
\bbP \left((\nu_{n,1},\ldots,\nu_{n,n}) \in d
      \mathbf z) \right) = \bar \bbP\left((\nu_{n,1},\ldots,\nu_{n,n}) \in d
      \mathbf z) \right),
\)
and consequently 
\[
\begin{split}
  \bar \bbP \bigg( (\xi_{n,1}, \ldots, \xi_{n,n}) \in d \mathbf x,&
    (\nu_{n,1},\ldots,\nu_{n,n}) \in d \mathbf z) \bigg)\\ &= \bar \bbP
  \left( (\xi_{n,1}, \ldots, \xi_{n,n}) \in d \mathbf x \right) \bar
  \bbP\left( (\nu_{n,1},\ldots,\nu_{n,n}) \in d \mathbf z\right).
\end{split}
\]

Now, for the $n$th row $\Xi_n$, we define the process
\begin{align}
  \label{eq:4}
  \G_n(t) := \frac{1}{n} \sum_{i=1}^{\floor{nt}} \nu_{n,i} -
  \sum_{i=1}^{\floor{nt}} \xi_{n,i} ~\forall t \in [0,1].
\end{align}
Then, the workload process is defined as 
\begin{align}
  \label{eq:5}
  \Phi(\G_n) := \G_n + \Psi(\G_n),
\end{align}
where $\Psi(\G_n)(\cdot) : = \sup_{0 \leq s \leq \cdot} (-\G_n(s))_+$ is the
Skorokhod regulator function.

\begin{proposition}\label{prop:workload}
  Conditional on the sequence of events $\sA_n := \{\sum_{i=1}^n \xi_{n,i} \leq 1 <
  \sum_{i=1}^{n+1} \xi_{n,i}\}$ $n \geq 1$, we have
  
  \noindent $(i)$ $\G_n \to 0~$ in $(\sD,U)~\bar \bbP-$a.s. as $n \to \infty$ and $\Phi(\G_n) \to 0$
  in $(\sD,U)~\bar \bbP-$a.s. as $n \to \infty$.

  \noindent $(ii)$ $\sn \G_n \To_{\bar\bbP} W - W^0$ in $(\sD,U)$ as $n \to \infty$ and $\sn
  \Phi(\G_n) \To_{\bar\bbP} \Phi(W - W^0)$ in $(\sD,U)$ as $n \to \infty$, where $W$ is a Brownian
  motion with zero drift and diffusion coefficient equal to $\sigma$,
  and $W^0$ is the Brownian bridge process defined in Proposition~\ref{thm:conditioned-fclt}.
\end{proposition}

\begin{proof}
  First, observe that if $\G_n \to 0$ in $(\sD,U)~\bar\bbP-$a.s. and $\sn \G_n
  \To_{\bar\bbP} W - W^0$ in $(\sD,U)$ as $n \to \infty$,
  then the convergence for the workload process follows automatically
  from the continuity of the Skorokhod regulator map, $\Phi(\cdot)$. 

  From the FSLLN \cite[Chapter 5]{ChYa01} we have,
  \begin{align*}
    \frac{1}{n} \sum_{i=0}^{\floor{n \cdot }} \nu_{n,i} \to
    e~\text{in }(\sD,U)~\bar\bbP-\text{a.s. as}~n \to \infty,
  \end{align*}
  and from the FCLT \cite[Chapter 5]{ChYa01},
  \begin{align*}
    \sn \left( \frac{1}{n} \sum_{i=0}^{\floor{n \cdot}} \nu_{n,i} - e
    \right) \Rightarrow_{\bar \bbP} W ~\text{as}~n \to \infty.
  \end{align*}
  On the other hand, Proposition~\ref{thm:conditioned-fslln} and
  Proposition~\ref{thm:conditioned-fclt} imply that
  \[
  \frac{1}{n}\sum_{i=1}^{\floor{n \cdot}} \xi_{n,l} \to
  e~\text{in}~(\sD,U)~\bar\bbP-\text{a.s. as } n \to \infty   
  \]
  and
  \[
  \sn \left( \frac{1}{n} \sum_{l=1}^{\floor{n \cdot}} \xi_{n,l} - e
  \right) \Rightarrow_{\bar \bbP} W^0~\text{in}~(\sD,U)~\text{as}~n\to\infty.
  \]
  Note that, for the latter result we also use the fact that $\mu_n
  \leq 1$ implies that
  \begin{align*}
    \sn \left( \floor{n \cdot} \m_n - e \right) &\leq \frac{1}{\sn}
                                                  \left( \floor{n
                                                  \cdot} - n e \right)
                                                  \to 0~\text{in}~(\sD,U)~\text{as}~n\to\infty.
  \end{align*}
  The continuity of the
  difference operator in the metric space $(\sD,U)$ implies that
  \[
  \G_n \to 0~\text{in}~(\sD,U)~\bar\bbP-\text{a.s. as}~n\to\infty,
  \]
  and
  \[
  \sn \G_n = \sn \left( \frac{1}{n}\sum_{l=0}^{\floor{nt}} \left(
      \nu_{n,l} - \xi_{n,l} \right) \right) \To_{\bar \bbP} \left(W
    - W^0 \right) ~\text{in}~(\sD,U)~\text{as}~n\to\infty.
  \]
$\QED$
\end{proof}

\subsection{Comparison with conventional heavy-traffic approximation}~\label{sec:compare}
We begin by observing that the Brownian bridge limit does not assume a
so-called ``heavy-traffic condition'' as is the case in standard
heavy-traffic approximations. The standard heavy-traffic condition fora sequence of queueing models indexed by $n$ having arrival
$\lambda_n$ and service rate $\mu_n$ assumes that 
\[
\sn \left( \frac{\l_n}{n} - \frac{\m_n}{n} \right) \to \theta
~\in~\bbR ~\text{as}~n\to\infty.
\]
If $\theta < 0$, then the load factor $\rho_n := \lambda_n/\m_n < 1$,
implying that the sequence of models are `underloaded'; that is, in the
long-term the workload process remains bounded. On the other
hand, if $\theta > 0$, then $\rho_n > 1$ and the sequence of models
are `overloaded'. In either case, however, $\lim_{n \to \infty} \rho_n =
1$. In other words, for large enough $n$ there are many arrivals, but
approximately a similar order of service completions as well. Note
that we are not assuming the limit diffusion approximation has a
steady state, hence considering $\theta \geq 0$ is acceptable in the
current analysis.

The workload approximation for a $GI/GI/1$ queue can be developed by
assuming $\l_n = n$ and $\m_n = n - \theta \sn$. Let the sequence of i.i.d. random variables $\{\nu_i,~i\geq 1\}$
and $\{\xi_i,~i\geq 1\}$ represent the service times and
inter-arrival times (respectively). Assume that $\mathbb E[\xi_1] = 1$
and $\mathbb E[\nu_1] = n/\mu_n = n/(n - \theta \sn)$ in the $n$th system. We define the
workload process $\Phi(\G_n)$ following~\eqref{eq:5}, with
\[
\G_n(t) = \left(\frac{1}{n} \sum_{i=0}^{\floor{nt}} \nu_{i} -
  \frac{nt}{\m_n} \right)
- \left(\frac{1}{n} \sum^{\floor{nt}}_{i=0} \xi_i - \frac{nt}{\l_n}\right) +
nt\left(\frac{1}{\m_n} - \frac{1}{\l_n}\right),
\]
where we assume that $\nu_0 = \xi_0 = 0$ a.s. Note that we do not 
define a triangular array anymore, since $\{\G_n\}$ can be
considered as a single sequence of stochastic processes. However,
there are versions of the heavy-traffic approximation where triangular
arrays can be used \cite{Br1998,Wi1998}.

Observe that 
$n^{3/2}\left(\frac{1}{\m_n} - \frac{1}{\l_n}\right) \to \theta$
as $n \to \infty$. This together with the FCLT \cite[Chapter 5]{ChYa01} implies that
\(
\sqrt n \G_n \Rightarrow \theta e + W - W'~\text{as}~n\to\infty,
\)
where $W$ and $W'$ are Brownian motion processes corresponding to
FCLTs for the service times and inter-arrival time sequences
(respectively) and $e : \mathbb R \to \mathbb R$ is the identity map,
as before. It follows that
\[
\sqrt n \Phi(\G_n) \Rightarrow \Phi(\theta e + W - W') ~\text{as}~ n\to\infty.
\] 
Thus, the diffusion
approximation in this case is a reflected Brownian motion. In
contrast, the limit process in
Proposition~\ref{prop:workload}~(ii) is a reflected Brownian bridge process explicitly capturing a
`depleting points effect,'  in the sense that
as the day progresses, there are fewer and fewer remaining jobs to
arrive; see \cite{BeHoJvL2015} for a rigorous definition. This effect is a
consequence of conditioning on the number of arrivals in the horizon.

From an operational analysis perspective, the `depletion of points' effect also implies that the increments of the workload
process display long-range correlations (if $\Theta(n)$ arrivals occur
in $[0,t)$, then there is necessarily few arrivals in the remaining
time). These effects are not present
in the standard $GI/GI/1$ heavy-traffic analysis. Furthermore, these effects can
affect operational decisions, and the choice of model is crucial and
application dependent. On the other hand,
from a simulation perspective, the main results in this paper suggest
approximations that can be easily simulated in the conduct of
`what if' type simulation analyses of $GI/GI/1$ queueing models. For
instance, it is of interest to ask what the distribution of the
workload is going to be at time $T (=1)$ conditioned on
$\mathcal A_n$. Simulating the workload process in
Proposition~\ref{prop:workload} to compute this is quite straightforward.

\bibliography{refs-queueing}
\bibliographystyle{apt}
\end{document}